\documentclass[10pt]{amsart}

\usepackage{mathscinet}

\usepackage{amsthm, amssymb}
\usepackage[leqno]{amsmath}
\usepackage{caption}
\usepackage{subcaption}
\usepackage{relsize}
\usepackage{booktabs}
\usepackage{hyperref}

\usepackage{tikz-cd}
\usepackage{tikz}
\usetikzlibrary{3d, matrix,decorations.pathreplacing,calc, patterns}

\theoremstyle{plain}
\newtheorem{theorem}{Theorem}[section]
\newtheorem{lemma}[theorem]{Lemma}
\newtheorem{proposition}[theorem]{Proposition}

\theoremstyle{definition}
\newtheorem{definition}[theorem]{Definition}
\newtheorem{example}[theorem]{Example}
\newtheorem{remark}[theorem]{Remark}

\numberwithin{equation}{section}
\allowdisplaybreaks[1]

\newcommand{\C}{\mathbb{C}}
\newcommand{\Z}{\mathbb{Z}}
\newcommand{\ep}{\varepsilon}

\usepackage{graphicx}

\makeatletter
\newcommand*\bigcdot{\mathpalette\bigcdot@{.5}}
\newcommand*\bigcdot@[2]{\mathbin{\vcenter{\hbox{\scalebox{#2}{$\m@th#1\bullet$}}}}}
\makeatother

\newcommand{\Lie}{{\text{Lie}}}

\makeatletter
\newcommand{\leqnomode}{\tagsleft@true\let\veqno\@@leqno}
\newcommand{\reqnomode}{\tagsleft@false\let\veqno\@@eqno}
\makeatother

\newcommand{\flag}{{\mathcal{F} \ell}}

\newcommand{\Lam}[2]{{\Lambda^{{(#1)}}_{{#2}}}}
\newcommand{\A}[2]{{A^{{(#1)}}_{{#2}}}}
\newcommand{\X}[2]{{X^{{(#1)}}_{{#2}}}}

\newcommand{\xii}[2]{{\xi^{{(#1)}}_{{#2}}}}
\newcommand{\bolda}[3]{{\mathbf{a}^{{(#1)}}_{{{#2}}, {{#3}}}}}
\newcommand{\Cstar}{{\C^{\ast}}}

\newcommand{\defi}[1]{{\textit{#1}}}

\begin{document}

\title[Generic torus orbit closures in flag {B}ott manifolds]{Generic torus orbit closures \\ in flag {B}ott manifolds}

\author{Eunjeong Lee}
\address{Center for Geometry and Physics, Institute for Basic Science (IBS), Pohang 37673, Korea}
\email{eunjeong.lee@ibs.re.kr}

\author{Dong Youp Suh}
\address{Department of Mathematical Sciences
	\\ KAIST \\ 291 Daehak-ro Yuseong-gu \\ Daejeon 34141 \\ Korea}
\email{dysuh@kaist.ac.kr}

\thanks{Lee and Suh were partially supported by Basic Science 
	Research Program through the National Research Foundation of Korea (NRF) 
	funded by the Ministry of  Science, ICT \& Future Planning 
	(No. 2016R1A2B4010823). Lee was partially supported by IBS-R003-D1.}
\date{\today}

\setcounter{tocdepth}{2} 

\begin{abstract}
In this article the generic torus orbit closure in a flag Bott manifold is shown to be a non-singular toric variety, and its fan structure is
explicitly calculated.
\end{abstract}

\maketitle

\section{Introduction}
On the full flag manifold $\flag(\mathbb C^{n+1})$ there is an effective action of complex torus~$(\mathbb C^\ast)^n$.
The generic torus orbit closure, which is the closure of a generic $(\mathbb{C}^{\ast})^n$-orbit, in the full flag manifold is well-known to be a non-singular toric  variety, called the permutohedral variety. The fan  of the permutohedral variety consists of Weyl chambers of a Lie group of $A_n$-type as its maximal cones.
Note that the closure of an arbitrary $(\Cstar)^n$-orbit is  known to be normal hence a toric variety, see~\cite[Proposition 4.8]{CaKu20}, but non-singularity of it is not determined in general.

Generic torus orbit closures in a generalized flag manifold $G/P$ are studied in~\cite{FlHa91} and~\cite{Dabr96}, and
arbitrary orbit closures in Grassmannian manifolds are studied in~\cite{GeGoMaSe87}, \cite{GeSe87}, \cite{BuTe18}. Furthermore generic torus orbit closures in  Schubert varieties are studied in \cite{LeeMasuda18}.

In \cite{KLSS18}, the notion of flag Bott manifold is introduced as a generalization of both full flag manifolds and Bott manifolds.
In fact, a flag Bott manifold $F_m$ is the total space of an $m$-sequence of iterated fiber bundles whose fibers are full flag manifolds $\flag(\mathbb C^{n_j +1})$ for $j=1,\ldots, m$, and there is an effective action of  complex torus ${\mathbf H}$ of rank $n=n_1+\cdots+n_m$. 
Therefore, it would be interesting to know when a torus orbit closure of $F_m$ is a non-singular toric variety and to determine its fan structure.

Certain flag Bott manifolds are constructed from  generalized Bott manifolds $B_m$, and such manifolds are called the associated flag Bott manifolds to $B_m$. It is shown in \cite{KLSS18} that the generic torus orbit closure in the associated flag Bott manifold to $B_m$ is
a non-singular toric variety, and such toric variety can be obtained from $B_m$ through a sequence of blow-ups. 

In this article, we consider the generic torus orbit closure $X$ of an arbitrary flag Bott manifold $F_m$.
The toric variety $X$ is the total space of an $m$-sequence of iterated fiber bundles with permutohedral varieties as its fibers.
We calculate the fan of $X$ in Theorem~\ref{main_thm}. As a consequence we can see that $X$ is a non-singular toric variety,
see Proposition~\ref{prop_X_is_smooth}.

\section{Flag Bott Manifolds}
\label{sec_flag_Bott_manifolds}
In this section we recall flag Bott manifolds from~\cite{KLSS18} and consider their orbit space construction.
Let $M$ be a complex manifold and $E$ an $n$-dimensional holomorphic vector bundle over $M$. 
Recall from \cite[p. 282]{BoTu82} that 
the \defi{associated flag bundle} 
$\flag(E) \to M$ is obtained from $E$ by replacing each fiber
$E_p$ by the full flag manifold $\flag(E_p)$.

\begin{definition}[{\cite[Definition 2.1]{KLSS18}}]\label{def_fBT}
	A \defi{flag Bott tower} $F_{\bullet} = \{ F_j \mid 0 \leq j \leq m\}$ of
	height $m$ (or an \defi{$m$-stage flag Bott tower}) is a sequence,
	\[
	\begin{tikzcd}
	F_m \arrow[r, "p_m"] & F_{m-1} \arrow[r, "p_{m-1}"] & \cdots \arrow[r, "p_2"]
	& F_1 \arrow[r, "p_1"] & F_0 = \{\text{a point}\},
	\end{tikzcd}
	\]
	of manifolds $F_j = \flag\left(\bigoplus_{k=1}^{n_j+1} \xii{j}{k} \right)$
	where $\xii{j}{k}$ is a holomorphic 
	line bundle over $F_{j-1}$ for each $1 \leq k \leq n_j+1$
	and $1\leq j \leq m$.
	We call $F_j$ the \defi{$j$-stage flag Bott manifold} of the
	flag Bott tower. 
\end{definition}

The full flag manifold $\flag(\C^{n+1})=:\flag(n+1)$ is a flag Bott manifold, and the product of flag manifolds $\flag(n_1+1) \times \cdots \times \flag(n_m+1)$ is a flag Bott manifold. Also an $m$-stage Bott manifold, which is a smooth projective toric variety, is an $m$-stage flag Bott manifold, see~\cite{GrKa94}.

We call two flag Bott towers $F_{\bullet}$ and $F_{\bullet}'$ are \defi{isomorphic} if there is a collection of holomorphic diffeomorphisms $\varphi_j \colon F_j \to F_j'$ which commute with the projections $p_j \colon F_j \to F_{j-1}$ and $p_j' \colon F_j' \to F_{j-1}'$ for all $1\leq j\leq m$.

Suppose that $F_{\bullet}$ is an $m$-stage flag Bott tower. 
To describe the orbit space construction of a flag Bott manifold, we first define the right action $\Phi_j$. Suppose that $B_{GL(n)}$ is the set of upper triangular matrices in $GL(n):=GL(n,\mathbb{C})$ for $n \in \Z_{>0}$. Also let $H_{GL(n)}$ be the set of diagonal matrices in $GL(n)$. 
For positive integers $n$ and $n'$, let $A$ be an integer matrix of size $(n+1) \times (n'+1)$ whose row vectors are $\mathbf{a}_1,\dots,\mathbf{a}_{n+1} \in \Z^{n'+1}$, i.e.,
\[
A = \begin{bmatrix}
\mathbf{a}_1 \\ \mathbf{a}_2 \\ \vdots \\ \mathbf{a}_{n+1}
\end{bmatrix}.
\]
Since the character group $\chi(H_{GL(n'+1)})$ is isomorphic to $\Z^{n'+1}$, the matrix $A$ defines a homomorphism $H_{GL(n'+1)} \to H_{GL(n+1)}$ given by
\[
h \mapsto \text{diag}(h^{\mathbf{a}_1},\dots,h^{\mathbf{a}_{n+1}}).
\]
Here $h=\text{diag}(h_1,\dots,h_{n'+1})$ is an element of $H_{GL(n'+1)}$ and
$h^{\bf a} := h_1^{\mathbf{a}(1)} \cdots h_{n'+1}^{\mathbf{a}(n'+1)}$ for
$\mathbf{a} = (\mathbf{a}(1),\dots, \mathbf{a}(n'+1)) \in \Z^{n'+1}$.
By composing the canonical projection $\Upsilon \colon B_{GL(n'+1)} \to H_{GL(n'+1)}$ which ignores entries not on the diagonals, the matrix~$A$ induces a homomorphism 
\[
\Lambda(A) \colon B_{GL(n'+1)} \to H_{GL(n+1)}. 
\]

Suppose that $(\A{j}{\ell})_{1 \leq \ell < j \leq m} \in \prod_{1 \leq \ell < j \leq m} M_{(n_j+1) \times (n_{\ell}+1)}(\mathbb{Z})$ is 
a sequence of integer matrices. 
We define a right action $\Phi_j$ of $B_{GL(n_1+1)} \times \cdots \times B_{GL(n_j+1)}$ on the product $GL(n_1+1) \times \cdots \times GL(n_j+1)$ by
\begin{equation}\label{eq_def_of_Phi_j}
\begin{split}
\Phi_j&((g_1,g_2,\dots,g_j), (b_1,b_2,\dots,b_j)) \\
& := 	(g_1b_1,
\Lam{2}{1}(b_1)^{-1}g_2b_2,
\Lam{3}{1}(b_1)^{-1}\Lam{3}{2}(b_2)^{-1}g_3b_3,\ldots,\\
& \quad \quad \quad
\Lam{j}{1}(b_1)^{-1}\Lam{j}{2}(b_2)^{-1} \cdots \Lam{j}{j-1}(b_{j-1})^{-1}
g_jb_j)
\end{split}
\end{equation}
where $\Lam{j}{\ell} := \Lambda(\A{j}{\ell})$ for $1 \leq \ell < j \leq m$. 
Then the action $\Phi_j$ is free and proper, see~\cite[Lemma 2.7]{KLSS18}. Hence
we have complex manifolds 
\[
(GL(n_1+1) \times \cdots \times GL(n_j+1))/\Phi_j \quad \text{ for } 1\leq j \leq m. 
\]
These manifolds are actually flag Bott manifolds (see~\cite[Proposition 2.8]{KLSS18}), and every flag Bott manifold can be obtained by this construction.
\begin{proposition}[{\cite[Proposition 2.11]{KLSS18}}]
	Let $F_{\bullet}$ be a flag Bott tower of height $m$. Then there is a sequence of integer matrices 
	\[
	(\A{j}{\ell})_{1 \leq \ell < j \leq m} \in \prod_{1 \leq \ell < j \leq m} M_{(n_j+1) \times (n_{\ell}+1)}(\mathbb{Z})
	\] 
	such that
	$F_{\bullet}$ is isomorphic to 
	\[
	\{(GL(n_1+1) \times \cdots \times GL(n_j+1))/\Phi_j \mid 0\leq j \leq m\}
	\]
	as flag Bott towers.
\end{proposition}
\begin{remark}
We notice that the sequence of integer matrices  
$(\A{j}{\ell})_{1 \leq \ell < j \leq m} \in \prod_{1 \leq \ell < j \leq m} M_{(n_j+1) \times (n_{\ell}+1)}(\mathbb{Z})$ 
are associated to the first Chern classes of line bundles $\xi^{(j)}_k$ in the construction of a flag Bott manifold. 
To be more precisely, suppose that $\bolda{j}{k}{\ell}$ be the $k$th row vector of the matrix $\A{j}{\ell}$. Then the set of vector $\{\bolda{j}{k}{1},\dots,\bolda{j}{k}{j-1}\}$ determines the first Chern class of the line bundle $\xi^{(j)}_k$. For more details, see~\cite[\S 2]{KLSS18}.
\end{remark}
\begin{example}\label{example_fBT_2_stage}
Suppose that $n_1 = 2$ and $n_2 = 1$. Consider 
$\A{2}{1} = \begin{bmatrix}
c_1 & c_2 & 0 \\ 0 & 0 & 0
\end{bmatrix}$.
Then the following Bott tower $\{F_j \mid 0 \leq j \leq 2\}$ is isomorphic to $\{(GL(n_1+1) \times \cdots \times GL(n_j+1))/\Phi_j \mid 0 \leq j \leq 2\}$ as flag Bott towers.
\[
\begin{tikzcd}[row sep = 0.2cm]
\flag(\xi(c_1,c_2,0) \oplus \underline{\C}) \arrow[d, equal] \rar & \flag(3) \arrow[d, equal] \rar & \{\text{a point}\} \arrow[d, equal] \\
F_2 & F_1 & F_0
\end{tikzcd}
\]
The line bundle $\xi(c_1,c_2,0)$ over $F_1$ is $(GL(3) \times \C)/B_{GL(3)}$, where the right action of $B_{GL(3)}$ is defined by
\[
(g,v) \cdot b = (gb, b_{11}^{-c_1} b_{22}^{-c_2} v)
\]
for $g \in GL(3)$ and $b = (b_{ij}) \in B_{GL(3)}$. 
\end{example}

\section{Generic torus orbit closures in flag Bott manifolds}
Let $F_m$ be an $m$-stage flag Bott manifold, and let $\mathbb H = H_{GL(n_1+1)} \times \cdots \times H_{GL(n_m+1)}$.
As introduced in \cite[\S 3.1]{KLSS18}, the \defi{canonical action} of $\mathbb H$ on $F_j$ is defined to be
\[
(h_1,\dots,h_m) \cdot [g_1,\dots,g_j] := [h_1g_1,\dots,h_jg_j] \quad
\text{ for } 1\leq j \leq m.
\]
Then $q_j \colon F_j \to F_{j-1}$ is an $\mathbb H$-equivariant fiber bundle. 
Note that this action is not effective. 
If we write $h_j = \text{diag}(h_{j,1},\dots,h_{j,n_j+1}) \in GL(n_j+1)$, then the subtorus
\[
\mathbf H := \{(h_1,\dots,h_m) \in \mathbb H \mid
h_{1,n_1+1} = \cdots = h_{m,n_m+1} = 1\} \cong (\Cstar)^{n}
\]
acts effectively on $F_m$, where
\[
n:= n_1 + \cdots + n_m.
\]

In order to consider the closure of torus orbit with respect to the canonical action, we define a generic element in $F_m$. Let $g = (g_{ij})$ be an element in $GL(n+1)$. For an ordered sequence $1 \leq i_1 < i_2 < \cdots < i_k \leq n+1$,
we define the Pl\"{u}cker coordinate
\[
X_{i_1,\dots,i_k}(g) := \det((g_{i_p,p})_{p=1,\dots,k}).
\]
\begin{definition}[{\cite[Definition 5.4]{KLSS18}}]
We call an element $g \in GL(n+1)$ is \defi{generic} if $X_{i_1,\dots,i_k}(g)$ is nonzero for any $k \in [n+1]$ and an ordered sequence $1 \leq i_1 < \cdots < i_k \leq n+1$. 
A point $[g_1,\dots,g_m] \in F_m$ is \defi{generic} if $g_j \in GL(n_j+1)$ is generic for all $1\leq j \leq m$.
A \defi{generic torus orbit} in $F_m$ is the $\mathbf{H}$-orbit of a generic point. 
\end{definition}
The above definition of generic elements and generic torus orbits can be found in~\cite{FlHa91}, \cite{Klya95}, and \cite{Dabr96}.
The closure $X_n$ of a generic torus orbit in the flag manifold~$\flag(n+1)$ is a smooth projective toric variety called the \defi{permutohedral variety}, see~\cite{Klya85}. 
We recall the fan $\Sigma_{X_n} \subset \mathbb{R}^{n}$ of the permutohedral variety $X_n$ for the later use. The rays in $\Sigma_{X_n}$ are parametrized by the nonempty proper subsets of $[n+1]$:
\[
\Sigma_{X_n}(1) \stackrel{1-1}{\longleftrightarrow} \{ S \mid \emptyset \subsetneq S \subsetneq [n+1] \}
\]
More precisely, for a nonempty proper subset $S$ of $[n+1]$, the corresponding ray $\rho_S$ is generated by
\[
u_S := \begin{cases}
\displaystyle\sum_{s \in S} \varepsilon_s & \text{ if } n+1 \notin S, \\
\displaystyle-\sum_{s \in [n+1] \setminus S} \varepsilon_s & \text{ otherwise},
\end{cases}
\]
where $\{\varepsilon_1,\dots,\varepsilon_n\}$ is the standard basis vector of $\mathbb{R}^n$. Hence there are $2^{n+1}-2$ many rays in $\Sigma_{X_n}$, see Figure~\ref{fig_ray}.
And the maximal cones in $\Sigma_{X_n}$ are  indexed by the set of proper chain of nonempty proper subsets of $[n+1]$. For a given proper chain 
\[
S_{\bullet} : \emptyset \subsetneq S_1 \subsetneq S_2 \subsetneq \cdots \subsetneq S_n \subsetneq [n+1]
\] 
of nonempty proper subsets, we have the corresponding maximal cone
\[
\text{Cone}(u_{S_1},u_{S_2},\dots,u_{S_n}).
\]
See Figure~\ref{fig_m_cones}.
Therefore $|\Sigma_{X_n}(n)| = (n+1)!$. 
\begin{figure}
\begin{subfigure}[b]{0.5\textwidth}
	\centering
\begin{tikzpicture}
	\draw[->] (0,0)--(0,1) node[above] {$u_{\{2\}}$};
	\draw[->] (0,0) -- (1,0) node[right] {$u_{\{1\}}$};
	\draw[->] (0,0) -- (1,1) node[right] {$u_{\{1,2\}}$};
	\draw[->] (0,0) -- (-1,0) node[left] {$u_{\{2,3\}}$};
	\draw[->] (0,0) -- (0,-1) node[below] {$u_{\{1,3\}}$};
	\draw[->] (0,0) -- (-1,-1) node[below] {$u_{\{3\}}$};
\end{tikzpicture}
\caption{Ray generators in $\Sigma_{X_2}$.}
\label{fig_ray}
\end{subfigure}%
\begin{subfigure}[b]{0.5\textwidth}
	\centering
\begin{tikzpicture}
\fill[pattern color = black!10!white, pattern = vertical lines] (0,0)--(2,0)--(2,2)--cycle;
\fill[pattern color = black!10!white, pattern = north east lines] (0,0) -- (-2,0) -- (-2,2)--(0,2) -- cycle;
\fill[pattern color = black!10!white, pattern = horizontal lines] (0,0)--(2,2)--(0,2)--cycle;
\fill[pattern color = black!10!white, pattern = north east lines] (0,0)--(2,0) -- (2,-2) -- (0,-2) -- cycle;
\fill[pattern color = black!10!white, pattern = horizontal lines] (0,0) -- (0,-2) -- (-2,-2) --cycle;
\fill[pattern color = black!10!white, pattern = vertical lines] (0,0) -- (-2,-2)--(-2,0) --cycle;

\draw[gray] (0,0) -- (0,2);
\draw[gray] (0,0) -- (2,0);
\draw[gray] (0,0) -- (2,2);
\draw[gray] (0,0) -- (-2,0);
\draw[gray] (0,0) -- (-2,-2);
\draw[gray] (0,0) -- (0,-2);

\node at (0.75,1.5) {\scriptsize $\{2\} \subsetneq \{1,2\}$};
\node at (1.25,0.3) {\scriptsize $\{1\} \subsetneq \{1,2\}$};
\node at (1.25,-1) {\scriptsize $\{1\} \subsetneq \{1,3\}$};
\node at (-0.75,-1.5) {\scriptsize $\{3\} \subsetneq \{1,3\}$};
\node at (-1.25,-0.3) {\scriptsize $\{3\} \subsetneq \{2,3\}$};
\node at (-1.25,1) {\scriptsize $\{2\} \subsetneq \{2,3\}$};

\draw[thick,->] (0,0)--(0,1) ;
\draw[thick,->] (0,0) -- (1,0) ;
\draw[thick,->] (0,0) -- (1,1) ;
\draw[thick,->] (0,0) -- (-1,0) ;
\draw[thick,->] (0,0) -- (0,-1) ;
\draw[thick,->] (0,0) -- (-1,-1) ;

\end{tikzpicture}
\caption{Maximal cones in $\Sigma_{X_2}$.}
\label{fig_m_cones}
\end{subfigure}
\caption{Fan $\Sigma_{X_2}$.}
\end{figure}
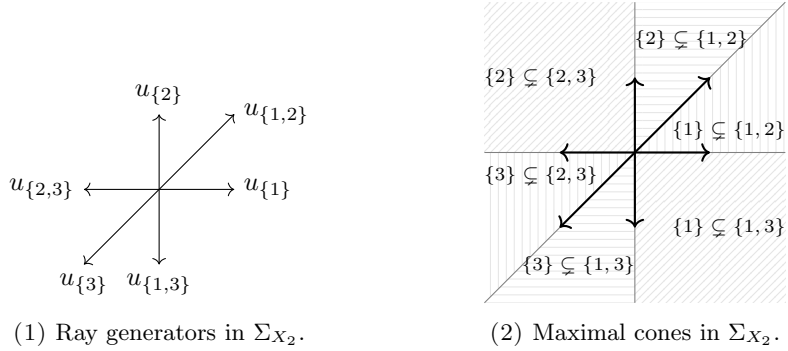

Suppose that $F_m$ is an $m$-stage flag Bott manifold. Considering the effective canonical $\mathbf H$-action, each fiber of a bundle $F_j \to F_{j-1}$ has the restricted $(\Cstar)^{n_j}$-action, and its orbit closure of a generic point is the permutohedral variety $X_{n_j}$. Hence the closure of a generic torus orbit of the torus $\mathbf H$ in $F_m$ is an iterated permutohedral varieties bundles. 
Now we describe the fan of a generic torus orbit closure in $F_m$ in the following theorem
whose proof will be given in Section~\ref{sec_proof_of_main_thm}.
\begin{theorem}\label{main_thm}
Let $F_m$ be a flag Bott manifold determined by the sequence of integer matrices $(\A{j}{\ell})_{1 \leq \ell < j \leq m} \in \prod_{1 \leq \ell < j \leq m} M_{(n_j+1) \times (n_{\ell}+1)}(\mathbb{Z})$.
Then the fan $\Sigma \subset \mathbb{R}^{n}$ of a generic torus orbit closure $X$ in $F_m$ is described as follows:
\begin{enumerate}
	\item the rays in $\Sigma$ are parametrized by 
	\[
	\{(\ell,S) \mid \emptyset \subsetneq S \subsetneq [n_{\ell}+1], 1 \leq \ell \leq m\}.
	\] 
	For $(\ell,S)$, the corresponding ray is generated by
	\[
	u^{\ell}_S := \begin{cases}\displaystyle
	\sum_{s \in S} \varepsilon_{\ell,s}  
	- \sum_{p = \ell+1}^m \sum_{k=1}^{n_{p}+1}
	((\A{p}{\ell})_{k,d+1}+\cdots+(\A{p}{\ell})_{k,n_{\ell}+1})\varepsilon_{p,k}
	& \text{ if } n_{\ell}+1 \notin S, \\
	\displaystyle -\sum_{s \in [n_{\ell}+1] \setminus S} \varepsilon_{\ell,s}
	 + \sum_{p=\ell+1}^m \sum_{k=1}^{n_p+1}((\A{p}{\ell})_{k,1} + \cdots+ (\A{p}{\ell})_{k,d})\varepsilon_{p,k} & \text{ otherwise},
	\end{cases}
	\]
	where $d = |[n_{\ell}+1]\setminus S|$ 
	and $\{\varepsilon_{\ell,k}\}_{1\leq k \leq n_{\ell},1 \leq \ell \leq m}$ is the standard basis vector in $\mathbb{R}^{n} \cong \Lie(\mathbf T)$. Here we set $\ep_{1,n_1+1} = \cdots = \ep_{m,n_m+1} = 0$.
	\item The maximal cones in $\Sigma$ are indexed by the sequences of proper chains of subsets
	\[\{(S^1_{\bullet},\dots,S^m_{\bullet}) \mid S^{\ell}_{\bullet} \colon \emptyset \subsetneq S_1^{\ell} \subsetneq S_2^{\ell} \subsetneq  \cdots \subsetneq S_{n_{\ell}}^{\ell} \subsetneq [n_{\ell}+1],1 \leq \ell \leq m\}. 
	\]
	For $(S^1_{\bullet},\dots,S^m_{\bullet})$, the corresponding maximal cone is defined to be
	\[
	\textup{Cone}\left(\bigcup_{\ell=1}^m\{u^{\ell}_{S_1^{\ell}},\dots,u^{\ell}_{S_{n_{\ell}}^{\ell}}\}\right).
	\]	
\end{enumerate}
\end{theorem}

We notice that if $F_m$ is determined by the sequence of integer matrices such that each $\A{j}{\ell}$ has nonzero value only on the first column, then $F_m$ is an \defi{associated flag Bott manifold} to a generalized Bott manifold (see, for more details, \cite[\S 4]{KLSS18}). The fan of generic torus orbit closure in an associated flag Bott manifold has been computed in~\cite[Theorem 5.5]{KLSS18}, and Theorem~\ref{main_thm} extends this result for considering any flag Bott manifold.

\begin{example}
Suppose that $F_2$ is a $2$-stage flag Bott manifold given in Example~\ref{example_fBT_2_stage}. Then the fan $\Sigma \subset \mathbb{R}^{3}$ of the generic torus orbit closure in $F_2$ has eight rays with the ray vectors:
\begin{align*}
u^1_{\{1\}} &= (1,0,0), & u^1_{\{2\}}&=(0,1,0), & 
u^1_{\{3\}}&=(-1,-1,c_1+c_2), \\
u^1_{\{1,2\}} &= (1,1,-c_2), & 
u^1_{\{2,3\}} &= (-1,0,c_1), & u^1_{\{1,3\}} &= (0,-1,c_1), \\
u^2_{\{1\}} &= (0,0,1), & u^2_{\{2\}} &= (0,0,-1). &  & \\
\end{align*}
Moreover the fan $\Sigma$ has twelve maximal cones:
\begin{align*}
&\text{Cone}(u^1_{\{1\}}, u^1_{\{1,2\}}, u^2_{\{1\}}), &
&\text{Cone}(u^1_{\{1\}}, u^1_{\{1,3\}}, u^2_{\{1\}}), &
&\text{Cone}(u^1_{\{2\}}, u^1_{\{1,2\}}, u^2_{\{1\}}), \\
&\text{Cone}(u^1_{\{2\}}, u^1_{\{2,3\}}, u^2_{\{1\}}), &
&\text{Cone}(u^1_{\{3\}}, u^1_{\{1,3\}}, u^2_{\{1\}}), &
&\text{Cone}(u^1_{\{3\}}, u^1_{\{2,3\}}, u^2_{\{1\}}), \\
&\text{Cone}(u^1_{\{1\}}, u^1_{\{1,2\}}, u^2_{\{2\}}), &
&\text{Cone}(u^1_{\{1\}}, u^1_{\{1,3\}}, u^2_{\{2\}}), &
&\text{Cone}(u^1_{\{2\}}, u^1_{\{1,2\}}, u^2_{\{2\}}), \\
&\text{Cone}(u^1_{\{2\}}, u^1_{\{2,3\}}, u^2_{\{2\}}), &
&\text{Cone}(u^1_{\{3\}}, u^1_{\{1,3\}}, u^2_{\{2\}}), &
&\text{Cone}(u^1_{\{3\}}, u^1_{\{2,3\}}, u^2_{\{2\}}).
\end{align*}
\end{example}

\begin{example}\label{example_F3_ray_vectors}
Suppose that $F_3$ is a $3$-stage flag Bott manifold determined by 
\[
\A{2}{1} = \begin{bmatrix}
x_{11} & x_{12} & 0 \\ 
x_{21} & x_{22} & 0 \\
0 & 0 & 0 
\end{bmatrix},\quad
\A{3}{1} = \begin{bmatrix}
y_1 & y_2 & 0 \\ 0 & 0 & 0
\end{bmatrix}, \quad
\A{3}{2} = \begin{bmatrix}
z_1 & z_2 & 0  \\ 0 & 0 & 0
\end{bmatrix}.
\]
Then the fan $\Sigma \subset \mathbb{R}^5$ of the generic torus orbit closure in $F_3$ has fourteen rays which are generated by the following vectors:
\begin{align*}
u^1_{\{1\}} &= (1,0,0,0,0), & u^1_{\{2\}} &= (0,1,0,0,0), \\ 
u^1_{\{3\}} &= (-1,-1, x_{11}+x_{12},x_{21}+x_{22}, y_1+y_2), & u^1_{\{1,2\}} &= (1,1,-x_{12},-x_{22},-y_2), \\ 
u^1_{\{2,3\}} &= (-1,0,x_{11},x_{21},y_1), & u^1_{\{1,3\}} &= (0,-1,x_{11},x_{21},y_1),\\
u^2_{\{1\}} &= (0,0,1,0,0), & u^2_{\{2\}} &= (0,0,0,1,0), \\
u^2_{\{3\}} &= (0,0,-1,-1,z_1+z_2), & u^2_{\{1,2\}} &= (0,0,1,1,-z_2), \\
u^2_{\{2,3\}} &= (0,0,-1,0,z_1), & u^2_{\{1,3\}} &= (0,0,0,-1,z_1), \\
u^3_{\{1\}} &= (0,0,0,0,1), & u^3_{\{2\}} &= (0,0,0,0,-1).
\end{align*}
\end{example}

\begin{proposition}\label{prop_X_is_smooth}
All generic torus orbit closures in a flag Bott manifold $F_m$ are isomorphic non-singular toric varieties.
\end{proposition}
To give a proof of Proposition~\ref{prop_X_is_smooth}, we use the following lemma:
\begin{lemma}[{\cite[Corollary of Theorem 3.3]{Dabr96}}]\label{lemma_permutohedral_smooth}
	The permutohedral variety $X_n$ is smooth, i.e., 
	for a proper chain $S_{\bullet} \colon \emptyset \subsetneq S_1 \subsetneq S_2 \subsetneq \cdots \subsetneq S_n \subsetneq [n+1]$ 
	of nonempty proper subsets of $[n+1]$, the determinant of the matrix
	\begin{equation}\label{eq_matrix_permutohedron}
	[ u_{S_1} \ u_{S_2} \ \cdots \ u_{S_n}]
	\end{equation}
	is $\pm 1$. 
\end{lemma}
\begin{proof}[Proof of Proposition~\ref{prop_X_is_smooth}]
	Let $\Sigma \in \mathbb{R}^{n}$ be the fan of a generic torus orbit closure $X$ in $F_m$ described in Theorem~\ref{main_thm}.
	To prove the claim, it is enough to show that every maximal cones in $\Sigma$ are smooth, i.e., the set of ray generators of each maximal cone forms a basis of $\mathbb{Z}^{n}$. 
	For a maximal cone indexed by $(S^1_{\bullet},\dots, S^m_{\bullet})$, consider the matrix whose column vectors are the corresponding ray generators: 
	\begin{equation}\label{eq_matrix_smooth}
	\left[u^1_{S^1_1} \ \cdots \ u^1_{S^1_{n_1}} \ \cdots \ u^m_{S^m_1} \ \cdots \ u^m_{S^m_{n_m}} \right]. 
	\end{equation}
	Then the matrix in~\eqref{eq_matrix_smooth} is a block lower triangular matrix whose sizes of blocks are $n_1,\dots,n_m$. Moreover, diagonal blocks has the form in~\eqref{eq_matrix_permutohedron}. Hence we have that the determinant of the matrix in~\eqref{eq_matrix_smooth} is
	\[
	\det\left(\left[u_{S^1_1} \ \cdots \ u_{S^1_{n_1}}\right]\right) \cdot
	\det\left(\left[u_{S^2_1} \ \cdots \ u_{S^2_{n_2}}\right]\right) \cdots 
	\det\left(\left[u_{S^m_1} \ \cdots \ u_{S^m_{n_m}}\right]\right) = \pm 1
	\]
	by Lemma~\ref{lemma_permutohedral_smooth}.
	Here $\{u_{S^{\ell}_1},\dots,u_{S^{\ell}_{n_{\ell}}}\}$ is the set of ray generators of the maximal cone in the fan of the permutohedral variety $X_{n_{\ell}}$ indexed by the proper chain $\emptyset \subsetneq S^{\ell}_1 \subsetneq \cdots \subsetneq S^{\ell}_{n_{\ell}} \subsetneq [n_{\ell}+1]$ for $1 \leq \ell \leq m$. Hence the non-singularity of $X$ follows.
	
	Since the ray vectors $u^{\ell}_S$ are determined independently of the choices of generic points, the fan of all
	generic torus orbit closures are identical. Therefore they are all isomorphic as smooth toric varieties.
\end{proof}
\begin{example}
	Suppose that $F_3$ is a $3$-stage flag Bott manifold defined in Example~\ref{example_F3_ray_vectors}. Consider a maximal cone $\sigma$ indexed by
	\[
	(\{2\} \subset \{2,3\}, \{2\} \subset \{1,2\}, \{2\}).
	\]
	Then the corresponding ray generators form the following matrix: 
	\[
	\left[u^1_{\{2\}} \ u^1_{\{2,3\}} \ u^2_{\{2\}} \ u^2_{\{1,2\}} \ u^3_{\{2\}}\right]
	= \begin{bmatrix}
	0 & -1 & 0 & 0 & 0 \\
	1 & 0 & 0 &0 & 0\\
	0 & x_{11} & 0 &1 &0\\
	0 & x_{21} & 1 & 1 & 0\\
	0 & y_1  & 0 & -z_2 & -1
	\end{bmatrix}.
	\]
	We have that
	\[
	\det \begin{bmatrix}
	0 & -1 & 0 & 0 & 0 \\
	1 & 0 & 0 &0 & 0\\
	0 & x_{11} & 0 &1 &0\\
	0 & x_{21} & 1 & 1 & 0\\
	0 & y_1  & 0 & -z_2 & -1
	\end{bmatrix}
	 = \det \begin{bmatrix}
	 0 & -1 \\ 1 & 0 
	 \end{bmatrix}
	 \det \begin{bmatrix}
	 0 & 1 \\ 1 &  1
	 \end{bmatrix}
	 \det
	 \begin{bmatrix}
-1
	 \end{bmatrix}
	= 1,
	\]
	so that the maximal cone $\sigma$ is smooth.
\end{example}

\section{Proof of Theorem~\ref{main_thm}}
\label{sec_proof_of_main_thm}

In this section we prove Theorem~\ref{main_thm} 
using the combinatorial structure of the fan of toric varieties bundles and
tangential representations around fixed points. 
To describe the fan $\Sigma$ of the torus orbit closure in $F_m$, we first recall the following on equivariant fiber bundle of toric varieties:
\begin{proposition}[{\cite[Proposition 7.3]{Oda78Torus} and \cite[\S 3.3]{CLS11}}]
	\label{prop_equivariant_bundle}
Let $\Sigma$, respectively $\Sigma'$, be a complete fan in $N_{\mathbb{R}} := N \otimes_{\mathbb Z} \mathbb{R}$, respectively $N'_{\mathbb{R}} := N' \otimes_{\mathbb Z} \mathbb{R}$, for some lattice $N$, respectively $N'$. Let $\overline{\varphi} \colon N \to N'$ be compatible with fans $\Sigma$ and $\Sigma'$, and let $N'' = \ker(\overline{\varphi})$.
Then $\varphi \colon X_{\Sigma} \to X_{\Sigma'}$ is an equivariant fiber bundle with the fiber $X_{\Sigma''}$, where $\Sigma'' = \{\sigma \in \Sigma \mid \sigma \subset N''_{\mathbb{R}} \} \subset N''_{\mathbb{R}}$, 
if and only if the following conditions are satisfied:
\begin{enumerate}
	\item $\overline{\varphi} \colon N \to N'$ is surjective.
	\item There exists a lifting $\widetilde{\Sigma}' \subset \Sigma$ of $\Sigma'$, i.e., for each $\sigma' \in \Sigma'$, there exists a unique $\widetilde{\sigma}' \in \widetilde{\Sigma}'$ such that $\overline{\varphi}$ induces a bijection
	\[
	\overline{\varphi} \colon \widetilde{\sigma}' \stackrel{\sim}{\longrightarrow} \sigma'.
	\]
	\item $\Sigma = \widetilde{\Sigma}' \bigcdot \Sigma''$, i.e., $\Sigma$ consists of cones 
	\[
	\sigma = \widetilde{\sigma}' + \sigma''
	\]
	with $\widetilde{\sigma}'$ and $\sigma''$ running through $\widetilde{\Sigma}'$ and $\Sigma''$. 
\end{enumerate}
\end{proposition}

The operation $\bigcdot$ is called \defi{join} in~\cite[III.1]{Ewald96Comb}. 
The closure of a generic torus orbit of the torus $\bf H$ in $F_m$ is an iterated permutohedral variety bundles. 
Hence we have the following lemma:
\begin{lemma}[{see~\cite[Lemma 5.11]{KLSS18}}]
	Let $F_m$ be an $m$-stage flag Bott manifold. Let $\Sigma$ be the fan of the closure of a generic torus orbit of the torus $\bf H$ in $F_m$. Then there are liftings $\widetilde{\Sigma}_{n_1},\dots,\widetilde{\Sigma}_{n_{m-1}}$ of fans of permutohedral varieties such that
	\[
	\Sigma = \widetilde{\Sigma}_{n_1} \bigcdot \cdots \bigcdot 	\widetilde{\Sigma}_{n_{m-1}} \bigcdot \Sigma_{n_m}. 
	\]
\end{lemma}
The above lemma proves that the combinatorial structure of the fan $\Sigma$ is given as in Theorem~\ref{main_thm}(2).
Now it is enough to show that the ray vectors are given as in Theorem~\ref{main_thm}(1).
To complete the proof, we use the tangential representations around fixed points which have been computed in~\cite[\S 3]{KLSS18}. 

Let $F_m$ be an $m$-stage flag Bott manifold. Recall from~\cite[Proposition 3.3]{KLSS18} that the fixed point set $F_m^{\bf H}$ can be identified with $\mathfrak{S}_{n_1+1} \times \cdots \times \mathfrak{S}_{n_m+1}$. More precisely, for permutations
$(v_1,\dots,v_m) \in \mathfrak{S}_{n_1+1} \times \cdots \times \mathfrak{S}_{n_m+1}$, the corresponding fixed point in $F_m$ is $[\dot{v}_1,\dots,\dot{v}_m]$ where $\dot{v}_j \in GL(n_j+1)$ is the column permutation matrix of $v_j$. Also we have that $F_m^{\bf H} = X^{\bf H}$.
Suppose that $\mathbb T$, respectively $\mathbf T$, is the maximal compact torus in $\mathbb H$, respectively $\mathbf H$.
Then, for every fixed point $v := [\dot{v}_1,\dots,\dot{v}_m]$, the weights of the isotropy representation $T_vF_m$ of~$\mathbb T$ are explicitly computed in~\cite[Proposition 3.5]{KLSS18}. 

We note that there is a one-to-one correspondence relation between the set $\mathfrak{S}_{n_1+1} \times \cdots \times \mathfrak{S}_{n_m+1}$ and the sequences of chains of subsets $\{(S^1_{\bullet},\dots,S^m_{\bullet})\}$ as follows. For a given $(v_1,\dots,v_m) \in \mathfrak{S}_{n_1+1} \times \cdots \times \mathfrak{S}_{n_m+1}$, we define
\begin{equation}\label{eq_def_of_A_ell_p}
S^{\ell}_{p} := \{v(n_{\ell}+2-p),\dots,v(n_{\ell}+1)\} \quad \text{ for }1\leq p\leq n_{\ell}, 1\leq \ell\leq m.
\end{equation}
Moreover, for a given maximal cone indexed by $(v_1,\dots,v_m)$, the adjacent maximal cones are indexed by permutations 
\begin{equation}\label{eq_adjacent_permutations}
(v_1,\dots,v_{j-1}, v_j \cdot s_i, v_{j+1},\dots,v_m)
\end{equation}
where $s_i$ is the transposition $(i,i+1)$ for $1 \leq i \leq n_j$ and $1 \leq j \leq m$.

For a smooth projective toric variety $X_{\Sigma}$ of complex dimension $n$, the weights of the isotropy representations around fixed points and the ray generators are closely related to each other. 
Let $\rho$ be a ray in $\Sigma(1)$ with $u_{\rho} =: u_{1}$ as its generating vector. Suppose that $\sigma = \text{Cone}(u_1,\dots,u_n)$ is a maximal cone containing $\rho$.  
Let $\tau_1,\dots,\tau_n$ be the codimension-one faces of the cone $\sigma$. More precisely, we put 
\[
\tau_j = \text{Cone}(u_1,\dots,\hat{u}_j,\dots,u_n). 
\]
Then $\tau_1$ is the unique codimension-one face of $\sigma$ which does not contain $\rho$. 
By the orbit-cone correspondence, these cones correspond to torus invariant spheres $S_1,\dots,S_n$ in $X_{\Sigma}$. Then these spheres meet at a point $p \in X_{\Sigma}$ which is exactly the fixed point corresponding $\sigma$.
Suppose that $w_1,\dots,w_n \in \Lie(\mathbf T)^{\ast}$ are weights of the isotropy representation $T_p X_{\Sigma}$, where $\mathbf T$ is the compact torus of dimension $n$. Let $H_i$ be the identity component of the kernel of the map 
	\[
	\exp(\sqrt{-1} w_i) \colon \mathbf T \to S^1.
	\] 
	Then, by reordering $w_1,\dots,w_n$ appropriately, one can see that the sphere $S_i$ is the connected component of $X_{\Sigma}^{H_i}$ containing $p$.
Moreover, we have the following:
\begin{lemma}[{\cite[Proposition 7.3.18]{BuPa15}}]
	Let $u_1 = u_{\rho}, w_1,\dots,w_n$ be as above. 
	Then we have the following relation:
	\[
	\langle w_i, u_{\rho} \rangle = 
	\begin{cases}
	1 & \text{ if } i = 1,\\
	0 & \text{ otherwise}.
	\end{cases}
	\]
\end{lemma}

The above lemma implies that the ray generators are completely determined by weights of isotropy representations around fixed points. Moreover, one can see that the computation of the ray generator $u_{\rho}$ is independent of the choice of a maximal cone containing $\rho$, see~\cite[Lemma 5.13]{KLSS18}.

We choose $1 \leq \ell \leq m$ and a nonempty proper subset $S$ of $[n_{\ell}+1]$. To compute the generator $u^{\ell}_S$ of the ray $\rho^{\ell}_S$, we consider a specific maximal cone $\sigma^{\ell}_S$ contains~$\rho^{\ell}_S$.
We set $S = \{s_1 < s_2 < \cdots < s_{n_{\ell}+1-d}\}$ and 
$[n_{\ell}+1] \setminus S = \{t_1 < t_2 < \cdots < t_d\}$. 
Let $v_{\ell,S}$ be a permutation in $\mathfrak{S}_{n_{\ell}+1}$ defined to be
\[
v_{\ell,S} = (t_1 \ t_2 \ \cdots \ t_d \ s_1 \ s_2 \ \cdots \ s_{n_{\ell}+1-d}).
\]
Then we choose a maximal cone $\sigma^{\ell}_S$  indexed by  
\begin{equation}\label{eq_def_of_v_ell}
v = (v_1,\dots,v_m) := (\underbrace{e,\dots,e}_{\ell-1}, v_{\ell,S}, e,\dots,e),
\end{equation}
so that $\sigma^{\ell}_S$ contains the ray $\rho^{\ell}_S$.

We recall from~\cite[Proposition 3.5 and Theorem 3.11]{KLSS18} that the weight $w^j_i$ corresponds to the codimension-one face 
of the cone $\sigma^{\ell}_S$ intersecting the maximal cone indexed by 
$(v_1,\dots,v_{j-1}, v_j \cdot s_i, v_{j+1},\dots,v_m)$ is computed as follows:
\begin{proposition}[{\cite[Proposition 3.5 and Theorem 3.11]{KLSS18}}]\label{prop_wji}
We have $w^j_i=r_{i+1} - r_i$ where $r_i$ is the $i$th row of the matrix
\[
\left[
\X{j}{1} \ \X{j}{2} \ \cdots \ \X{j}{j-1} \ B_j \ O \ \cdots \ O
\right].
\]
Here $\X{j}{\ell}$ is the matrix of size $(n_j+1) \times (n_{\ell}+1)$ defined by 
\begin{equation}\label{eq_def_of_Xjl}
\begin{split}
X^{(j)}_{\ell} &= \sum_{\ell < i_1 < \cdots < i_r < j }
\left(B_j A^{(j)}_{i_r} \right) \left(B_{i_r} A^{(i_{r})}_{i_{r-1}} \right) \cdots 
\left(B_{i_1} A^{(i_1)}_{\ell} \right)B_{\ell}\\
& \qquad + B_j \A{j}{\ell} B_{\ell}
\end{split}
\end{equation}
for $1 \leq \ell < j \leq m$,
and $B_j$ is the row permutation matrix corresponding to $v_j$, i.e.,
$B_j = (\dot{v}_j)^T$.
\end{proposition}

We note that the codimension-one face of the cone $\sigma^{\ell}_S$ which does not contain the ray $\rho^{\ell}_S$ corresponds to the weight $w^{\ell}_d$. Hence
to complete the proof, it is enough to show that the vector
\[
u^{\ell}_S = \begin{cases}\displaystyle
\sum_{s \in S} \varepsilon_{\ell,s}  
- \sum_{p = \ell+1}^m \sum_{k=1}^{n_{p}+1}
((\A{p}{\ell})_{k,d+1}+\cdots+(\A{p}{\ell})_{k,n_{\ell}+1})\varepsilon_{p,k}
& \text{ if } n_{\ell}+1 \notin S, \\
\displaystyle -\sum_{s \in [n_{\ell}+1] \setminus S} \varepsilon_{\ell,s}
+ \sum_{p=\ell+1}^m \sum_{k=1}^{n_p+1}((\A{p}{\ell})_{k,1} + \cdots+ (\A{p}{\ell})_{k,d})\varepsilon_{p,k} & \text{ otherwise}
\end{cases}
\]
in Theorem~\ref{main_thm} satisfies that
\begin{equation}\label{eq_wji_ulA}
\langle w^j_i, u^{\ell}_S \rangle = \begin{cases}
1 & \text{ if } j=\ell \text{ and } i = d, \\
0 & \text{ otherwise}.
\end{cases}
\end{equation}

	The weight $w^j_i$, computed in~\cite[Proposition 3.5]{KLSS18}, is an element of $\Lie(\mathbb T)^{\ast}$. Since we have
	\[
	\mathbf T = \{(t_1,\dots,t_m) \in \mathbb T \mid t_{1,n_1+1} = \cdots = t_{m,n_m+1} = 1\}
	\]
	where $t_j = (t_{j,1},\dots,t_{j,n_j+1}) \in (S^1)^{n_j+1}$, we regard 
	\[
	\ep_{1,n_1+1} = \cdots = \ep_{m,n_m+1} = 0
	\]
	to compute the pairing~\eqref{eq_wji_ulA}.

To prove the claim, we separate cases as $j < \ell$,
$j = \ell$, and $j > \ell$. 
\begin{enumerate}
	\item[\textbf{Case 1}] {$j < \ell$.}
	By Proposition~\ref{prop_wji},
	the weight vector $w^j_i \in \Lie(\mathbf T)^{\ast}$ is a linear combination of 
	{
		\def\OldComma{,}
		\catcode`\,=13
		\def,{%
			\ifmmode%
			\OldComma\discretionary{}{}{}%
			\else%
			\OldComma%
			\fi%
		}%
		$\ep_{1,1}^{\ast},\dots,\ep_{1,n_1}^{\ast},\dots,
		\ep_{j,1}^{\ast}, \dots, \ep_{j,n_j}^{\ast}$,
	}
where $\{\ep_{\ell,k}^{\ast}\}_{1 \leq k \leq n_{\ell}, 1 \leq \ell \leq m}$ is the dual of the standard basis vector.
	On the other hand,
	since $u^{\ell}_S$ is a linear combination of 
	{
		\def\OldComma{,}
		\catcode`\,=13
		\def,{%
			\ifmmode%
			\OldComma\discretionary{}{}{}%
			\else%
			\OldComma%
			\fi%
		}%
		$\ep_{\ell,1},\dots,
		\ep_{\ell,n_{\ell}},\dots,\ep_{m,1},\dots,\ep_{m, n_m}$} and $j < \ell$,
	their pairings always vanish.
	
	\item[\textbf{Case 2}] $j = \ell$.
	By Proposition~\ref{prop_wji},
	the weight vector $w^{\ell}_i \in \Lie(\mathbf T)^{\ast}$ is 
	a linear combination of 
	{
		\def\OldComma{,}
		\catcode`\,=13
		\def,{%
			\ifmmode%
			\OldComma\discretionary{}{}{}%
			\else%
			\OldComma%
			\fi%
		}%
		$\ep_{1,1}^{\ast},
		\dots,\ep_{1,n_1}^{\ast},\dots,\ep_{\ell,1}^{\ast}, \dots, 
		\ep_{\ell,n_{\ell}}^{\ast}$}. More precisely, we have that 
	\[
	w^{\ell}_i = (\ep_{\ell, v_{\ell,S}(i+1)})^{\ast}-
	(\ep_{\ell,v_{\ell,S}(i)})^{\ast} + \text{ other terms},
	\]
	where `other terms' are the terms of 
	$\ep_{p,k}^{\ast}$ for $p < \ell$ and
	$v_{\ell,S}$ is a permutation defined in \eqref{eq_def_of_v_ell}.
	Since the vector $u^{\ell}_S$ is a linear combination of 
	$\ep_{\ell,1},\dots,\ep_{\ell,n_{\ell}},\dots,
	\ep_{m,1},\dots,\ep_{m,n_m}$, we have
	\begin{equation}\label{eq_alpha_and_lambda}
	\langle
	w^{\ell}_i, u^{\ell}_S
	\rangle
	= \langle
	(\ep_{\ell, v_{\ell,S}(i+1)})^{\ast} - (\ep_{\ell,
		v_{\ell,S}(i)})^{\ast}, u^{\ell}_S
	\rangle.
	\end{equation}
	Because of the definition of permutation $v_{\ell,S}$, we have that 
	$v_{\ell,S}(i) \in S $ if and only if $i \geq d+1$. 
	Therefore for the case when $n_{\ell}+1 \notin S$, we have that
	the value $\langle (\ep_{\ell, v_{\ell,S}(i)})^{\ast}, u^{\ell}_S \rangle$ 
	equals to 
	$0$ if $i \leq d$, and $1$ otherwise.  
	Also for the case when $n_{\ell}+1 \in S$, we get that
	the pairing 
	$\langle (\ep_{\ell, v_{\ell,S}(i)})^{\ast}, u^{\ell}_S \rangle$
	is $-1$ if $i \leq d$ and $0$ otherwise. 
	
	By applying \eqref{eq_alpha_and_lambda} for 
	$n_{\ell} +1 \notin S$, we have the following:
	\[
	\langle w^{\ell}_i, u^{\ell}_S
	\rangle = \begin{cases}
	0-0 =0 & \text{ for } 1\leq i \leq d-1, \\
	1-0 = 1 & \text{ for } i = d,\\
	1-1=0& \text{ for } d+1 \leq i \leq n_{\ell}.
	\end{cases}
	\]
	Similarly, when $n_{\ell}+1 \in S$, we get 
	the following:
	\[
	\langle w^{\ell}_i, u^{\ell}_S \rangle
	= \begin{cases}
	-1-(-1) = 0 &\text{ for } 1\leq i \leq d-1, \\
	0 - (-1) = 1 & \text{ for } i = d, \\
	0-0=0 & \text{ for } d+1 \leq i  \leq n_{\ell}.
	\end{cases}
	\]
	\item[\textbf{Case 3}] $j > \ell$. 
	The matrix $\X{\ell}{j}$ in Proposition~\ref{prop_wji} is 
	\[
	\X{j}{\ell}=
	\sum_{\ell < i_1 < \cdots < i_r < j}
	\left( B_j \A{j}{i_r} \right)
	\left( B_{i_r} \A{i_r}{i_{r-1}} \right)
	\cdots
	\left( B_{i_1} \A{i_1}{\ell} \right)
	B_{\ell} + B_j \A{j}{\ell} B_{\ell}.
	\]
	Since $v_j = e$ for $j \neq \ell$, the matrix $\X{j}{\ell}$ can
	be written by
	\[
	\X{j}{\ell} =
	\left(\sum_{\ell < i_1 < \cdots < i_r < j}
	\A{j}{i_r} 
	\A{i_r}{i_{r-1}}
	\cdots
	\A{i_1}{\ell} + \A{j}{\ell}\right)
	B_{\ell}.
	\]
	Moreover, we have that
\begin{equation}\label{eq_Xjell_and_A}
\begin{split}
\X{j}{\ell}B_{\ell}^{-1} &=
\sum_{\ell < i_1 < \cdots < i_r < j}
\A{j}{i_r} 
\A{i_r}{i_{r-1}}
\cdots
\A{i_1}{\ell} + \A{j}{\ell}\\
&= 
\X{j}{j-1}\A{j-1}{\ell}
+ \cdots 
+ \X{j}{\ell+2} \A{\ell+2}{\ell} 
+ \X{j}{\ell+1} \A{\ell+1}{\ell}
+ \A{j}{\ell} \\
&= \sum_{p=\ell+1}^{j-1} \X{j}{p} \A{p}{\ell}
+ \A{j}{\ell}.
\end{split}
\end{equation}

For notational simplicity, we denote the $i$th row vector of the matrix $\X{j}{\ell}$ by $x^{(j)}_{\ell,i}$. 
Then the weight $r_i$ in Proposition~\ref{prop_wji} equals to
	\[
	x^{(j)}_{1,i} + \cdots + x^{(j)}_{j-1,i} + \ep_{j,i}^{\ast}
	\]
because $B_j$ is the identity matrix.
Hence the pairing
\begin{equation}\label{eq_pairing_X_and_A}
\begin{split}
&\left\langle r_i, \sum_{p=\ell+1}^{m} \sum_{k=1}^{n_p+1} (\A{p}{\ell})_{k,z}\ep_{p,k}  \right\rangle \\
&\qquad = \left\langle
x^{(j)}_{\ell+1,i} + \cdots + x^{(j)}_{j-1,i} + \varepsilon_{j,i}^*,
\sum_{p=\ell+1}^{m} \sum_{k=1}^{n_p+1} (\A{p}{\ell})_{k,z}\ep_{p,k} 
\right\rangle 
\end{split}
\end{equation}
is the $(i,z)$entry of the matrix $\sum_{p=\ell+1}^{j-1} \X{j}{p} \A{p}{\ell} + \A{j}{\ell}$ for $1 \leq z \leq n_{\ell}+1$.
By~\eqref{eq_Xjell_and_A}, the pairing~\eqref{eq_pairing_X_and_A} is same as the $(i,z)$entry of the matrix $\X{j}{\ell} B^{-1}_{\ell}$:
\begin{equation}\label{eq_pairing_XA_and_XB}
\left\langle
r_i,
\sum_{p=\ell+1}^{m} \sum_{k=1}^{n_p+1} (\A{p}{\ell})_{k,z}\ep_{p,k} 
\right\rangle 
= (\X{j}{\ell} B^{-1}_{\ell})_{i,z}.
\end{equation}

\begin{enumerate}
	\item[\textbf{Subcase 1}] $n_{\ell}+1 \notin S$. In this case,
we first note that
\begin{equation}\label{eq_subcase1}
\left\langle r_i, \sum_{s \in S} \ep_{\ell,s} \right\rangle =
\left\langle x^{(j)}_{\ell,i}, \sum_{s \in S} \ep_{\ell,s} \right\rangle
= \sum_{s \in S} (\X{j}{\ell})_{i,s}.
\end{equation}
Since $B_{\ell}^{-1}$ is the column permutation matrix for $v_{\ell,A}$, the right hand side of~\eqref{eq_subcase1} coincides with
\begin{equation}\label{eq_subcase1_1}
(\X{j}{\ell} B^{-1}_{\ell})_{i,d+1} + (\X{j}{\ell} B^{-1}_{\ell})_{i,d+2}+
\cdots + (\X{j}{\ell} B^{-1}_{\ell})_{i,n_{\ell}+1}.
\end{equation}
This is same as $i$th entry of the sum of $(d+1),\dots,(n_{\ell}+1)$th column vectors in $\X{j}{\ell} B_{\ell}^{-1}$. 
Hence we have the following:
\begin{align*}
\langle r_i, u^{\ell}_S \rangle 
&= \left\langle 
r_i, \sum_{s \in S} \varepsilon_{\ell,s}  
- \sum_{p = \ell+1}^m \sum_{k=1}^{n_{p}+1}
((\A{p}{\ell})_{k,d+1}+\cdots+(\A{p}{\ell})_{k,n_{\ell}+1})\varepsilon_{p,k}
\right \rangle \\
&= \sum_{s \in S} (\X{j}{\ell})_{i,s} \\
&\qquad - \left( (\X{j}{\ell} B_{\ell}^{-1})_{i,d+1} +\cdots + (\X{j}{\ell} B_{\ell}^{-1})_{i,n_{\ell}+1} \right) \quad (\text{by~}\eqref{eq_pairing_XA_and_XB} \text{ and }\eqref{eq_subcase1}) \\
&= 0 \quad (\text{by~}\eqref{eq_subcase1_1}).
\end{align*}
Since $w^j_i = r_{i+1} - r_i$, the pairing $\langle w^j_i, u^{\ell}_S \rangle$ vanishes.

\item[\textbf{Subcase 2}] $n_{\ell}+1 \in S$.
In this case,
we have 
\begin{equation}\label{eq_subcase2}
\left\langle r_i, -\sum_{s \in [n_{\ell}+1]\setminus S} \ep_{\ell,s}\right\rangle =
\left\langle x^{(j)}_{\ell,i}, -\sum_{s \in [n_{\ell}+1]\setminus S} \ep_{\ell,s} \right\rangle
= -\sum_{s \in [n_{\ell}+1]\setminus S} (\X{j}{\ell})_{i,s}.
\end{equation}
Since $B_{\ell}^{-1}$ is the column permutation matrix for $v_{\ell,A}$, the right hand side of~\eqref{eq_subcase2} coincides with
\begin{equation}\label{eq_subcase2_1}
-(\X{j}{\ell} B^{-1}_{\ell})_{i,1} - (\X{j}{\ell} B^{-1}_{\ell})_{i,2}-
\cdots - (\X{j}{\ell} B^{-1}_{\ell})_{i,d}.
\end{equation}
This is same as $i$th entry of the sum of $1$st,$\ldots,d$th column vectors in $-\X{j}{\ell} B_{\ell}^{-1}$. Therefore, using the similar argument to Subcase 1, one can see that
\begin{align*}
\langle r_i, u^{\ell}_S \rangle 
&= \left\langle 
r_i, -\sum_{s \in [n_{\ell}+1] \setminus S} \varepsilon_{\ell,s}
+ \sum_{p=\ell+1}^m \sum_{k=1}^{n_p+1}((\A{p}{\ell})_{k,1} + \cdots+ (\A{p}{\ell})_{k,d})\varepsilon_{p,k}
\right \rangle \\
&= -\sum_{s \in [n_{\ell}+1]\setminus S} (\X{j}{\ell})_{i,s} \\
&\qquad + \left( (\X{j}{\ell} B_{\ell}^{-1})_{i,1} +\cdots + (\X{j}{\ell} B_{\ell}^{-1})_{i,d} \right) \quad (\text{by~}\eqref{eq_pairing_XA_and_XB} \text{ and }\eqref{eq_subcase2}) \\
&= 0 \quad (\text{by~}\eqref{eq_subcase2_1}).
\end{align*}
Hence  the pairing $\langle w^j_i, u^{\ell}_S \rangle$ vanishes since $w^j_i = r_{i+1} - r_i$.
\end{enumerate}
\end{enumerate}
\providecommand{\bysame}{\leavevmode\hbox to3em{\hrulefill}\thinspace}
\providecommand{\MR}{\relax\ifhmode\unskip\space\fi MR }
\providecommand{\MRhref}[2]{%
	\href{http://www.ams.org/mathscinet-getitem?mr=#1}{#2}
}
\providecommand{\href}[2]{#2}

\end{document}